\documentclass[11pt,english]{article}
\usepackage[T1]{fontenc}
\usepackage[latin9]{inputenc}
\usepackage{geometry}
\usepackage{color}
\usepackage{array}
\usepackage{booktabs}
\usepackage{units}
\usepackage{url}
\usepackage{multirow}
\usepackage{amsmath}
\usepackage{amssymb}

\makeatletter

\providecommand{\tabularnewline}{\\}

\usepackage{babel}

\usepackage{amsthm}\usepackage{tikz}
\usepackage{url}

\usepackage{units}
\usepackage{comment}

\DeclareMathOperator{\GrL}{GrL}
\DeclareMathOperator{\Aut}{Aut}
\DeclareMathOperator{\diag}{diag}

\newcommand{\s}{\sigma}
\newcommand{\Z}{\mathbb{Z}}
\newcommand{\one}{\textnormal{\textbf{1}}}
\newcommand{\m}[1]{\textnormal{\textbf{#1}}}

\newcommand{\ceil}[1]{{\left\lceil #1 \right\rceil}}
\newcommand{\setbuild}[2]{\left\{#1\textnormal{ : }#2\right\}}

\newtheorem{thm}{Theorem}[section]
\newtheorem{prop}[thm]{Proposition}\newtheorem{cor}[thm]{Corollary}

\makeatother

\usepackage{babel}
\begin{document}

\title{Growing Graceful Trees}

\author{Edinah K. Gnang \thanks{Department of Applied Mathematics and Statistics, Johns Hopkins, egnang1@jhu.edu},
Isaac Wass \thanks{Department of Mathematics, Iowa State University icwass@iastate.edu}}
\maketitle
\begin{abstract}
We describe symbolic constructions for listing and enumerating graphs
having the same induced edge label sequence. We settle in the affirmative
R. Whitty's \cite{W} conjectured existence of determinantal constructions
for listing and enumerating gracefully labeled trees. We conclude
the paper with a description of a new graceful labeling algorithm.
\end{abstract}

\section{Introduction}

The Kotzig-Ringel-Rosa \cite{R64} conjecture, better known as the
\textit{Graceful Labeling Conjecture }(GLC), asserts that every tree
is graceful. This conjecture has spurred a large body of work extensively
surveyed by Gallian in \cite{Gal05}. Let $\Z_{n}$ denote the set
$\left[0,n\right)\cap\Z$. Every function $f:\Z_{n}\rightarrow\Z_{n}$
is associated with a \textit{functional directed graph} $G_{f}$ whose
vertex and edge sets are $\Z_{n}$ and $\setbuild{\left(i,f\left(i\right)\right)}{i\in\Z_{n}}$
respectively. Our approach is based upon a functional reformulation
of the GLC. Induced subtractive edge labels are absolute differences
of integers assigned to vertices spanning each edge. We adopt the
notation convention 
\[
f^{\left(k+1\right)}:=f\circ f^{\left(k\right)},\quad\forall\,\begin{array}{c}
f\in\Z_{n}^{\Z_{n}}\\
k>0
\end{array},
\]
where $f^{\left(0\right)}$ denotes the identity function noted id.
We define graceful labelings of functional directed graphs to be vertex
labelings which yield bijections which maps vertex labels to \emph{induced
subtractive edge labels}. Let S$_{n}\subset\Z_{n}^{\Z_{n}}$ denote
the symmetric group on $\Z_{n}$ and $f\in\Z_{n}^{\Z_{n}}$ then $G_{f}$
is graceful if 
\[
\max_{\sigma\in\text{S}_{n}}\left|\setbuild{\left|\s f(i)-\s(i)\right|}{i\in\Z_{n}}\right|=n.
\]
Consequently, $G_{f}$ is gracefully labeled if 
\[
\setbuild{\left|f(i)-i\right|}{i\in\Z_{n}}=\Z_{n}
\]
A functional tree $G_{f}$ on $\Z_{n}$ is a spanning functional directed
graph whose underlying $f\in\Z_{n}^{\Z_{n}}$ is such that $\left|f^{\left(n-1\right)}\left(\Z_{n}\right)\right|=1$.
The following proposition expresses a necessary and sufficient condition
for a functional directed graph to be graceful. 

\begin{prop}(Graceful expansion) For any positive integer $n$ and
any function $f\in\Z_{n}^{\Z_{n}}$, the following statements are
equivalent: 
\begin{enumerate}
\item[(i).]  $G_{f}$ is graceful.
\item[(ii).] There exists $\s,\gamma\in S_{n}$ and $p\in\left\{ 0,1\right\} ^{\Z_{n}}$
such that
\[
f(i)=\s\Big(\s^{\left(-1\right)}(i)+(-1)^{p\left(\s^{\left(-1\right)}(i)\right)}\gamma\left(\s^{\left(-1\right)}(i)\right)\Big),\quad\forall\,i\in\Z_{n}.
\]
\end{enumerate}
\end{prop}

\begin{proof} Recall that $G_{f}$ is graceful iff 
\[
\max_{\sigma\in\text{S}_{n}}\left|\setbuild{\left|\s f(i)-\s(i)\right|}{i\in\Z_{n}}\right|=n.
\]
That is, $G_{f}$ is graceful if and only if there exists $\s\in S_{n}$
such that $\setbuild{|\s f(i)-\s(i)|}{i\in\Z_{n}}=\Z_{n}$. Consequently,
there is a permutation $\gamma\in S_{n}$ such that $|\s f(i)-\s(i)|=\gamma(i)$
for each $i\in\Z_{n}$. Additionally, there is a function $p:\Z_{n}\rightarrow\{0,1\}$
such that $(-1)^{p(i)}\cdot|\s f(i)-\s(i)|=\s f(i)-\s(i)$. We obtain
the following: 
\begin{quote}
$G_{f}$ is graceful $\iff$ $\exists\,\s,\gamma\in S_{n}$ and $p:\Z_{n}\rightarrow\{0,1\}$
s.t. $\s f(i)-\s(i)=(-1)^{p(i)}\gamma(i),\forall i\in\Z_{n}$. 
\end{quote}
Solving for $f(i)$ completes the proof. \end{proof}

Let $\GrL(G_{f})$ denote the set $\setbuild{G_{\s f\s^{-1}}}{\s\in\nicefrac{S_{n}}{\Aut G_{f}},\;G_{\s f\s^{-1}}\text{ is gracefully-labeled}}$,
where $\Aut G_{f}$ is the automorphism group of $G_{f}$. That is,
$\GrL(G_{f})$ denotes the set of all gracefully labeled graphs obtained
via vertex relabeling of $G_{f}$. Clearly $\GrL(G_{f})$ is nonempty
if and only if $G_{f}$ is graceful.

The \textit{induced subtractive edge label sequence} of a graph refers
to the non-decreasing sequence of induced subtractive edges labels.
For example, consider the function $f:\Z_{6}\rightarrow\Z_{6}$ defined
by
\[
f\left(i\right)=\begin{cases}
0 & 0\leq i\leq3\\
3 & 4\leq i\leq5
\end{cases},
\]
 whose functional directed graph $G_{f}$ is depicted in Figure 1.
Note that $G_{f}$ is a functional tree since $f^{\left(2\right)}\left(\Z_{6}\right)=\left\{ 0\right\} $.

\begin{figure}
\begin{tikzpicture}
	\node (1) at (-0.5,1) {1};
	\node (2) at (-0.5,-1) {2};
	\node (0) at (1,0) {0};
	\node (3) at (2.5,0) {3};
	\node (4) at (4,1) {4};
	\node (5) at (4,-1) {5};
	
	\foreach \x/\y in {1/0,2/0,3/0,4/3,5/3} {
		\draw[thick,->] (\x)--(\y);
	}
	\draw[thick,->] (0) edge [out=105,in=50,looseness=5] (0);
\end{tikzpicture} \centering \caption{A functional tree on 6 vertices.}
\end{figure}
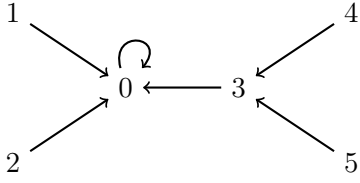
The edge set of $G_{f}$ is $E(G_{f})=\{(0,0),(1,0),(2,0),(3,0),(4,3),(5,3)\}$,
so the corresponding induced subtractive edge label sequence is $\left(0,1,1,2,2,3\right)$.

The GLC is easily verified for the families of functional star trees
associated with identically constant functions. This is seen from
the fact that the graph of the identically constant zero function
$f\in\mathbb{Z}_{n}^{\mathbb{Z}_{n}}$ is such that $G_{f}$ is gracefully
labeled and $\left|\text{GrL}\left(G_{f}\right)\right|=2$. One can
further show that $G_{f}$ has $\ceil{\frac{n}{2}}$ distinct subtractive
edge label sequences. Our main results are symbolic constructions
for listing and enumerating graphs having the same induced edge label
sequence. We settle in the affirmative R. Whitty's \cite{W} conjectured
existence of determinantal constructions for listing and enumerating
gracefully labeled trees. We conclude the paper with a description
of a new graceful labeling algorithm.

\section{Enumeration of gracefully labeled functional digraphs. }

There are $n!$ gracefully labeled undirected graphs on $\mathbb{Z}_{n}$
having $n$ edges. Unfortunately, very few such graphs are orrientable
into functional directed graphs. As a result it is more difficult
to enumerate gracefully labeled functional directed graphs. We derive
from the graceful expansion described in Proposition 1.1 an upper
bound for the number of gracefully labeled functional directed graphs.
Recall that $G_{f}$ is graceful if there exists a coset representative
$\sigma\in\text{\ensuremath{\nicefrac{\text{S}_{n}}{\Aut G_{f}}}}$
and a permutation $\gamma\in$ S$_{n}$ such that 
\[
f\left(i\right)\in\sigma\left(\left\{ \sigma^{\left(-1\right)}\left(i\right)-\gamma\sigma^{\left(-1\right)}\left(i\right),\,\sigma^{\left(-1\right)}\left(i\right)+\gamma\sigma^{\left(-1\right)}\left(i\right)\right\} \cap\Z_{n}\right),\quad\forall\;i\in\Z_{n},
\]
\[
\implies\sigma^{\left(-1\right)}f\sigma\left(i\right)\in\left\{ i-\gamma\left(i\right),\,i+\gamma\left(i\right)\right\} ,\quad\forall\;i\in\Z_{n}.
\]
Assume that $f\left(0\right)=0$, to ensure that the gracefully labeled
functional directed graph $G_{f}$ has no isolated vertices. Consequently,
the perrmutation $\gamma\in$ S$_{n}$ is such that $\gamma\left(0\right)=0$
and
\begin{equation}
\forall\,i\in\left[1,n\right)\cap\Z,\ \left|\left\{ i-\gamma\left(i\right),\,i+\gamma\left(i\right)\right\} \cap\Z_{n}\right|>0\Rightarrow\begin{cases}
\begin{array}{c}
\gamma\left(i\right)\le i\\
\mbox{ or }\\
\gamma\left(i\right)<n-i
\end{array}\forall\,i\in\Z_{n}\end{cases}.\label{Permutation condition}
\end{equation}
Alternatively, $G_{f}\in$ GrL$\left(G_{f}\right)$ if and only if
\[
f\left(i\right)=i+\left(-1\right)^{p\left(i\right)}\gamma\left(i\right),\quad\forall\,i\in\Z_{n},
\]
where $p\in\left\{ 0,1\right\} ^{\Z_{n}}$.

\begin{prop} For every positive integer $n>2$, we have 
\[
\left|\left\{ \gamma\in\mbox{S}_{n}:\gamma\left(0\right)=0\text{ and }\begin{array}{c}
\forall\,i\in\Z_{n}\backslash\left\{ 0\right\} ,\\
\Z_{n}\cap\left\{ i-\gamma\left(i\right),i+\gamma\left(i\right)\right\} \ne\emptyset
\end{array}\right\} \right|=\left(\left\lfloor \frac{n-1}{2}\right\rfloor !\right)\left(\left\lceil \frac{n-1}{2}\right\rceil !\right).
\]
The corresponding sequence appears in the OEIS database as \cite{OEIS A010551}

\end{prop}

\begin{proof} The proof argument follows from the graceful expansion
of $f\in\Z_{n}^{\Z_{n}}$ associated with a gracefully labeled functional
directed graph $G_{f}$ expressed by the addition setup :
\begin{center}
\begin{tabular}{rr@{\extracolsep{0pt}.}lr@{\extracolsep{0pt}.}lr@{\extracolsep{0pt}.}lr@{\extracolsep{0pt}.}lr@{\extracolsep{0pt}.}lr@{\extracolsep{0pt}.}lr@{\extracolsep{0pt}.}lr@{\extracolsep{0pt}.}l}
 & \multicolumn{2}{c}{$0$} & \multicolumn{2}{c}{$1$} & \multicolumn{2}{c}{$\cdots$} & \multicolumn{2}{c}{$i$} & \multicolumn{2}{c}{$\cdots$} & \multicolumn{2}{c}{$n-1$} & \multicolumn{2}{c}{} & \multicolumn{2}{c}{first row}\tabularnewline
$+$ & \multicolumn{2}{c}{} & \multicolumn{2}{c}{} & \multicolumn{2}{c}{} & \multicolumn{2}{c}{} & \multicolumn{2}{c}{} & \multicolumn{2}{c}{} & \multicolumn{2}{c}{} & \multicolumn{2}{c}{}\tabularnewline
\multirow{1}{*}{} & \multicolumn{2}{c}{\multirow{1}{*}{$0$}} & \multicolumn{2}{c}{$\left(-1\right)^{p\left(1\right)}\cdot\gamma\left(1\right)$} & \multicolumn{2}{c}{$\cdots$} & \multicolumn{2}{c}{$\left(-1\right)^{p\left(i\right)}\cdot\gamma\left(i\right)$} & \multicolumn{2}{c}{$\cdots$} & \multicolumn{2}{c}{$\left(-1\right)^{p\left(n-1\right)}\cdot\gamma\left(n-1\right)$} & \multicolumn{2}{c}{} & \multicolumn{2}{c}{second row}\tabularnewline
\midrule
\midrule 
$=$ & \multicolumn{2}{c}{$0$} & \multicolumn{2}{c}{$f\left(1\right)$} & \multicolumn{2}{c}{$\cdots$} & \multicolumn{2}{c}{$f\left(i\right)$} & \multicolumn{2}{c}{$\cdots$} & \multicolumn{2}{c}{$f\left(n-1\right)$} & \multicolumn{2}{c}{} & \multicolumn{2}{c}{}\tabularnewline
 & \multicolumn{2}{c}{} & \multicolumn{2}{c}{} & \multicolumn{2}{c}{} & \multicolumn{2}{c}{} & \multicolumn{2}{c}{} & \multicolumn{2}{c}{} & \multicolumn{2}{c}{} & \multicolumn{2}{c}{}\tabularnewline
\end{tabular}
\par\end{center}

Setting $f\left(0\right)=0$, to ensure that the corresponding gracefully
labeled graph $G_{f}$ has no isolated vertex implies that $f\left(n-1\right)=0$.
Consequently for any valid $\gamma$ there is a unique choice for
$\left(-1\right)^{p\left(n-1\right)}\gamma\left(n-1\right)$ namely
\[
\left(-1\right)^{p\left(n-1\right)}\gamma\left(n-1\right)=-\left(n-1\right).
\]
Following this first assignment, there are exactly two possible choices
for a column entry of the second row ( of the addition setup ) whose
absolute value equals $\left(n-2\right)$. These two possible choices
are prescribed by 
\begin{equation}
\begin{array}{ccc}
\left(-1\right)^{p\left(1\right)}\gamma\left(1\right) & = & n-2\\
 & \text{or}\\
\left(-1\right)^{p\left(n-2\right)}\gamma\left(n-2\right) & = & -\left(n-2\right)
\end{array}.\label{Choices}
\end{equation}
Clearly, only one of the two choices in Eq. (\ref{Choices}) occurs
in any valid $\gamma$. Following choices of column entries of the
second row whose magnitudes are $\left(n-1\right)$ and $\left(n-2\right)$
, there are three choices for a column entry of the second row whose
absolute value equals $\left(n-3\right)$. Note that all possible
choices ( not accounting for the entry of magnitude equal to $\left(n-2\right)$
) are prescribed by
\begin{equation}
\begin{array}{ccc}
\left(-1\right)^{p\left(1\right)}\gamma\left(1\right) & = & n-3\\
 & \text{or}\\
\left(-1\right)^{p\left(2\right)}\gamma\left(2\right) & = & n-3\\
 & \text{or}\\
\left(-1\right)^{p\left(n-3\right)}\gamma\left(n-3\right) & = & -\left(n-3\right)\\
 & \text{or}\\
\left(-1\right)^{p\left(n-2\right)}\gamma\left(n-2\right) & = & -\left(n-3\right)
\end{array}.\label{Choice2}
\end{equation}
However either the first or the last choice in Eq. (\ref{Choice2})
will be unavailable per the previous choice made for the entry of
magnitude $\left(n-2\right)$, leaving precisely three remaining choices
for the entry of magnitude $\left(n-3\right)$ as claimed. Similarly,
following the three choices made for column entries of the second
row of magnitudes $\left(n-1\right)$, $\left(n-2\right)$ and $\left(n-3\right)$,
there are four remaining choices for the a column entry of the second
row of the addition whose absolute value equals $\left(n-4\right)$.
All the possible choices (not accounting for the entries of magnitudes
$\left(n-2\right)$ and $\left(n-3\right)$ ) are prescribed by 
\begin{equation}
\begin{array}{ccc}
\left(-1\right)^{p\left(1\right)}\gamma\left(1\right) & = & n-4\\
 & \text{or}\\
\left(-1\right)^{p\left(2\right)}\gamma\left(2\right) & = & n-4\\
 & \text{or}\\
\left(-1\right)^{p\left(3\right)}\gamma\left(3\right) & = & n-4\\
 & \text{or}\\
\left(-1\right)^{p\left(n-4\right)}\gamma\left(n-4\right) & = & -\left(n-4\right)\\
 & \text{or}\\
\left(-1\right)^{p\left(n-3\right)}\gamma\left(n-3\right) & = & -\left(n-4\right)\\
 & \text{or}\\
\left(-1\right)^{p\left(n-2\right)}\gamma\left(n-2\right) & = & -\left(n-4\right)
\end{array}.\label{choice3}
\end{equation}
Two of the six possible choices prescribed in Eq. (\ref{choice3})
will be unavailable by the previous assignments made for column entries
of magnitudes $\left(n-2\right)$ and $\left(n-3\right)$. Thereby
leaving four choices for the column entry of magnitude $\left(n-4\right)$
as claimed. The argument proceeds similarly all the way up to the
choices for the column entry of the second row whose absolute value
equals $\left\lceil \frac{n-1}{2}\right\rceil $. These assignment
options account for the factorial factor $\left\lfloor \frac{n-1}{2}\right\rfloor !$.
Note that for each one of these choices, the sign is also known. Finally,
the last factorial factor arises from taking all possible permutations
of the remaining integers and thus completes the proof. \end{proof}
Let $\tau_{n}$ denote the number of gracefully labeled functional
directed graphs on $n$ vertices having no isolated vertices. As corollary
of Prop. 2.1 
\[
\left(\left\lfloor \frac{n-1}{2}\right\rfloor !\right)\left(\left\lceil \frac{n-1}{2}\right\rceil !\right)\,2\le\tau_{n}\le\left(\left\lfloor \frac{n-1}{2}\right\rfloor !\right)\left(\left\lceil \frac{n-1}{2}\right\rceil !\right)\:n\,2^{\left\lceil \frac{n-1}{2}\right\rceil }.
\]
The extra factor of $2$ in the lower bound accounts for the complementary
labeling involution functional map on $\Z_{n}^{\Z_{n}}$ prescribed
by
\[
f\mapsto n-1-f\left(n-1-\text{id}\right).
\]
For we know that the complementary labeling involution map preserves
graceful labelings. The extra factor of $n$ in the upper bound accounts
for alternative possible choices of the fixed point. Incidentally
the argument used to prove Proposition 2.1 describes an optimal algorithm
for constructing the set of permutations noted SP$_{n}$ ( used to
construct gracefully labeled functional directed graphs having no
isolated vertices ) defined by 
\begin{equation}
\text{SP}_{n}:=\left\{ g\in\text{S}_{2n-1}\subset\left(\left(-n,n\right)\cap\Z\right)^{\left(-n,n\right)\cap\Z}:\begin{array}{c}
g\left(-i\right)=-g\left(i\right)\\
\text{and }\forall\,i\in\left(-n,n\right)\cap\Z,\\
\text{id}+g\in\left(\left(-n,n\right)\cap\Z\right)^{\left(-n,n\right)\cap\Z}
\end{array}\right\} ,\label{Signed Permutations}
\end{equation}
Consequently, 
\[
\sum_{g\in\text{SP}_{n}}\prod_{i\in\Z_{n}}\mathbf{A}\left[i,i+g\left(i\right)\right]=\sum_{\begin{array}{c}
f\in\Z_{n}^{\Z_{n}}\\
f\left(0\right)=0\\
G_{f}\in\text{GrL}\left(G_{f}\right)
\end{array}}\prod_{i\in\Z_{n}}\mathbf{A}\left[i,f\left(i\right)\right].
\]
Whitty shows in \cite{W} that 
\[
\mathbf{A}\left[0,0\right]\det\left\{ \left(\boldsymbol{\Upsilon}-\boldsymbol{\Lambda}\right)\left[1:,1:\right]\right\} =\sum_{\begin{array}{c}
f^{\left(n-1\right)}\left(\Z_{n}\right)=\left\{ 0\right\} \\
G_{f}\in\text{GrL}\left(G_{f}\right)
\end{array}}\text{sgn}\left(\left|f-\text{id}\right|\right)\prod_{i\in\Z_{n}}\mathbf{A}\left[\min\left(i,f\left(i\right)\right),\max\left(i,f\left(i\right)\right)\right]
\]
\[
\text{where}
\]
\[
\forall\:0\le i,j<n,\ \begin{cases}
\begin{array}{ccc}
\boldsymbol{\Lambda}\left[i,j\right] & = & \mathbf{A}\left[\min\left(j-(n-1)+i-1,i\right),\max\left(j-(n-1)+i-1,i\right)\right]\\
\\
\boldsymbol{\Upsilon}\left[i,j\right] & = & \mathbf{A}\left[\min\left(i,(n-1)-j+i+1\right),\max\left(i,(n-1)-j+i+1\right)\right]
\end{array}.\end{cases}
\]
Whitty also conjectures in \cite{W} the existence of similar determinental
constructions whose terms are free of the signing factor $\text{sgn}\left(\left|f-\text{id}\right|\right)$

\section{Generatingfunctionology of induced edge labelings }

Motivated by Whitty's conjecture, we derive generating functions whose
coefficients enumerate functional directed graphs having the same
induced subtractive edge label sequence. The first construction follows
from the listing of functional directed graphs.

\begin{prop} For any $n\times n$ matrix $\mathbf{A}$ we have 
\begin{equation}
\det\left(\diag\left(\mathbf{A}\cdot\mathbf{1}_{n\times1}\right)\right)=\sum_{f\in\Z_{n}^{\Z_{n}}}\prod_{i\in\Z_{n}}\mathbf{A}\left[i,f\left(i\right)\right].\label{sum over functional digraphs}
\end{equation}
\end{prop}

\begin{proof} 
\[
\det\left(\diag\left(\mathbf{A}\cdot\mathbf{1}_{n\times1}\right)\right)=\prod_{i\in\Z_{n}}\left(\sum_{j\in\Z_{n}}\mathbf{A}\left[i,j\right]\right)=\sum_{f\in\Z_{n}^{\Z_{n}}}\prod_{i\in\Z_{n}}\mathbf{A}\left[i,f\left(i\right)\right].
\]
Thus completing the proof. \end{proof} \textbf{}\\
Let $\m{X}$ denote the symbolic $n\times n$ matrix where $\m{X}[i,j]=x^{(n+1)^{|i-j|}}$,
and we define the univariate polynomial 
\[
F_{\mathbf{X}}\left(x\right)=\det\left(\diag\left(\m{X}\cdot\one_{n\times1}\right)\right).
\]

\begin{cor} For $n\geq1$, the polynomial $F_{\mathbf{X}}\left(x\right)$
is the generating function whose coefficients enumerate the number
of distinct functional directed graphs on $n$ vertices with the same
induced subtractive edge label sequence. \end{cor}

\begin{proof} From the previous proposition $\det(\diag(\m{X}\cdot\one_{n}))=\underset{f\in\Z_{n}^{\Z_{n}}}{\sum}\underset{i\in\Z_{n}}{\prod}x^{\left(n+1\right)^{\left|i-f(i)\right|}}.$
Consider any functional directed graph with $b_{i}$ edges having
the edge label $i$ for each $i\in\Z_{n}$. Then the graph contributes
to the coefficient of the term with exponent $\sum\limits _{i\in\Z_{n}}b_{i}(n+1)^{i}$.
Every nonnegative integer has a unique decomposition in base $n+1$,
so $0\leq b_{i}\leq n$ implies that $\sum\limits _{i\in\Z_{n}}b_{i}(n+1)^{i}$
is uniquely determined by the induced subtractive edge label sequence
of the graph. Hence two graph contribute to the same coefficient if
and only if they have the same induced subtractive edge label sequence,
settling the proof. \end{proof}

\begin{prop}\label{prop-f_X-properties} The polynomial $F_{\mathbf{X}}\left(x\right)$
has the following properties:
\begin{itemize}
\item[i.] The lowest-degree term in $F_{\mathbf{X}}\left(x\right)$ is $x^{n}$. 
\item[ii.] If $n$ is even then $F_{\mathbf{X}}\left(x\right)$ has degree $2\sum\limits _{0<i<\frac{n}{2}}\left(n+1\right)^{\frac{n}{2}+i}$.\\
 Otherwise, $F_{\mathbf{X}}\left(x\right)$ has degree $(n+1)^{(n-1)/2}+2\sum\limits _{0<i<\frac{n-1}{2}}\left(n+1\right)^{\frac{n-1}{2}+i}$. 
\item[iii.] Asymptotically, $F_{\mathbf{X}}\left(x\right)$ has $O(\frac{4^{n}}{\sqrt{n}})$
non-vanishing coefficients. 
\end{itemize}
\end{prop}

\begin{proof} Recall that every nonzero term in $F_{X}(x)$ has an
exponent of the form $\sum\limits _{i\in\Z_{n}}b_{i}(n+1)^{i}$, where
$b_{i}$ is the number of edges with induced subtractive edge label
$i$ in a particular functional directed graph.

Note this forces $\sum\limits _{i\in\Z_{n}}b_{i}=n$.

(i) The minimal exponent possible is $n(n+1)^{0}=n$, which is associated
with the induced subtractive edge label sequence containing only zeros.
This sequence is realized by the functional directed graph associated
with the identity function. $\square$\\

(ii) The maximal exponent possible is obtained by choosing $f(i)$
that maximizes $|i-f(i)|$ for each $i$. This is equivalent to maximizing
$\max(i-f(i),f(i)-i)$, which is clearly maximized by $\max(i,n-1-i)$.
Note $i$ is larger when $i>n-1-i$, which simplifies to $i>\frac{n-1}{2}$.

If $n$ is odd, this means the induced subtractive edge label sequence
is 
\[
\left\{ \frac{n-1}{2},\frac{n-1}{2}+1,\frac{n-1}{2}+1,\frac{n-1}{2}+2,\frac{n-1}{2}+2,\dots,n-3,n-3,n-2,n-2,n-1,n-1\right\} .
\]
This yields an exponent of $\left(n+1\right)^{\frac{n-1}{2}}+\sum\limits _{0<i<\frac{n-1}{2}}2\left(n+1\right)^{\frac{n-1}{2}+i}$.

If $n$ is even the sequence is 
\[
\left\{ \frac{n}{2},\frac{n}{2},\frac{n}{2}+1,\frac{n}{2}+1,\dots,n-3,n-3,n-2,n-2,n-1,n-1\right\} .
\]
This yields an exponent of $\sum\limits _{0<i<\frac{n}{2}}2\left(n+1\right)^{\frac{n}{2}+i}$.
$\square$

This construction corresponds to a functional balanced double-star
with a two-cycle between the central vertices.\\

(iii) Consider the term with exponent $\sum\limits _{i\in\Z_{n}}b_{i}(n+1)^{i}$,
where $0\leq b_{i}\leq n$ for each $i$. For this term to have a
non-vanishing coefficient, it is necessary that $n=\sum\limits _{i\in\Z_{n}}b_{i}$,
as a functional directed graph on $n$ vertices has $n$ edges, and
each edge must be counted by some $b_{i}$. Thus a simple upper bound
for the number of non-vanishing terms is just the number of nonnegative
integer solutions to the equation $n=\sum\limits _{i\in\Z_{n}}b_{i}$,
of which there are $\binom{2n-1}{n}$. This is asymptotically equal
to $\frac{4^{n}}{2\sqrt{\pi n}}$ via Stirling's approximation. \end{proof}

The next proposition refines the construction to obtain the generating
function whose coefficients enumerate the number of distinct functional
trees which have the same induced subtractive edge label sequence.

Let $\m{X}$ denote the symbolic $n\times n$ matrix where $\m{X}[i,j]=x^{n^{|i-j|}}$,
and we define the univariate polynomial
\[
P_{\mathbf{X}}\left(x\right)=\sum\limits _{\underset{\left|f^{\left(n-1\right)}\left(\Z_{n}\right)\right|=1}{f\in\Z_{n}^{\Z_{n}}}}\prod\limits _{i\in\Z_{n}}\m{X}[i,f(i)]
\]

\begin{prop} $P_{\mathbf{X}}\left(x\right)$ is the generating function
whose coefficients enumerate the number of distinct functional trees
on $n$ vertices with the same induced subtractive edge label sequence.
\end{prop}

\begin{proof} The summation is taken over functions $f:\Z_{n}\rightarrow\Z_{n}$
subject to $\left|f^{\left(n-1\right)}\left(\Z_{n}\right)\right|=1$,
so the only functions considered are those associated with functional
trees.

Consider any functional tree with $b_{i}$ edges with induced subtractive
edge label $i$ for each $i\in\Z_{n}$. Note that functional trees
have exactly one root, so $b_{0}=1$. The functional tree contributes
to the coefficient of the term with exponent $\sum\limits _{i\in\Z_{n}}b_{i}n^{i}$.

Since $b_{0}=1$, then $0\leq b_{i}\leq n-1$ for each $i\in\Z_{n}$.
Every non-negative integer has a unique decomposition in base $n$,
so $0\leq b_{i}\leq n-1$ implies that $\sum\limits _{i\in\Z_{n}}b_{i}n^{i}$
is uniquely determined by the induced subtractive edge label sequence
of the functional tree. Hence two functional trees contribute to the
same coefficient if and only if they have the same induced subtractive
edge label sequence, settling the proof. \end{proof} For a $n\times n$
matrix $\mathbf{M}$, let
\[
\mathbf{M}\left[\begin{array}{c}
i_{0},\cdots,i_{t},\cdots,i_{k-1}\\
i_{0},\cdots,i_{t},\cdots,i_{k-1}
\end{array}\right]
\]
where $0\le i_{0}<\cdots<i_{j}<\cdots<i_{k-1}<n$ denote the $k\times k$
sub-matrix formed by retaining only the rows and columns of $\mathbf{M}$
indexed by $\left\{ i_{j}:j\in\Z_{k}\right\} $. In particular
\[
\mathbf{M}\left[:n-1,:n-1\right]=\mathbf{M}\left[\begin{array}{c}
0,\cdots,i,\cdots,n-2\\
0,\cdots,i,\cdots,n-2
\end{array}\right]
\]
\begin{prop} 
\[
P_{\mathbf{X}}\left(x\right)=\sum\limits _{i\in\Z_{n}}\mathbf{X}\left[i,i\right]\det\left\{ \left(\diag\left(\mathbf{X}\cdot\one_{n\times1}\right)-\mathbf{X}\right)\left[\begin{array}{c}
0,\cdots,i-1,i+1,\cdots,n-1\\
0,\cdots,i-1,i+1,\cdots,n-1
\end{array}\right]\right\} .
\]
 \end{prop}

\begin{proof}The result follows from Tutte's Directed Matrix Tree
Theorem (TDMTT). We reproduce here for the readers convenience Zeilberger's
combinatorial proof \cite{Z85} of TDMTT. For an arbitrary directed
graph $G$ on $n$ vertices ( allowing for loop edges ), let the corresponding
symbolic adjacency matrix $\mathbf{A}_{G}$ be given by
\[
\mathbf{A}_{G}\left[i,j\right]=\begin{cases}
\begin{array}{cc}
a_{ij} & \text{ if }\left(i,j\right)\in E\left(G\right)\\
0 & \text{otherwise}
\end{array} & \:\forall\:0\le i,j<n.\end{cases}
\]
To show that
\[
\left(\sum_{\begin{array}{c}
f\in\Z_{n}^{\Z_{n}}\\
\left|f^{\left(n-1\right)}\left(\Z_{n}\right)\right|=1
\end{array}}\prod_{i\in\Z_{n}}\mathbf{A}_{G}\left[i,f\left(i\right)\right]\right)=
\]
 
\[
\sum_{i\in\left[0,n\right)\cap\mathbb{Z}}\mathbf{A}_{G}\left[i,i\right]\,\det\left\{ \left(\text{diag}\left(\mathbf{A}_{G}\cdot\mathbf{1}_{n\times1}\right)-\mathbf{A}_{G}\right)\left[\begin{array}{c}
0,\cdots,i-1,i+1,\cdots,n-1\\
0,\cdots,i-1,i+1,\cdots,n-1
\end{array}\right]\right\} ,
\]
it suffices to show that for an arbitrary symbolic $\left(n+1\right)\times\left(n+1\right)$
symbolic matrix $\mathbf{A}$,
\[
\mathbf{A}\left[n,n\right]\det\left\{ \left(\text{diag}\left(\mathbf{A}\cdot\mathbf{1}_{n+1\times1}\right)-\mathbf{A}\right)\left[:n,:n\right]\right\} =\mathbf{A}\left[n,n\right]\sum_{\begin{array}{c}
f\in\Z_{n+1}^{\Z_{n+1}}\\
f^{\left(n\right)}\left(\Z_{n+1}\right)=\left\{ n\right\} 
\end{array}}\prod_{i\in\Z_{n}}\mathbf{A}\left[i,f\left(i\right)\right]
\]
Let $\mathbf{t}$ denote an $n\times1$ vector whose entries are given
by 
\[
\mathbf{t}\left[i\right]=\sum_{j\in\Z_{n+1}}\mathbf{A}\left[i,j\right],\quad\forall\,i\in\Z_{n}.
\]
Recall that 
\[
\mathbf{A}\left[n,n\right]\det\left(\text{diag}\left(\mathbf{t}\right)-\mathbf{A}\left[:n,:n\right]\right)=
\]
\[
\mathbf{A}\left[n,n\right]\sum_{0\le k\le n}\sum_{0\le i_{0}<\cdots<i_{k-1}<n}\left(-1\right)^{k}\det\left(\mathbf{A}\left[\begin{array}{ccccc}
i_{0} & \cdots & i_{j} & \cdots & i_{k-1}\\
i_{0} & \cdots & i_{j} & \cdots & i_{k-1}
\end{array}\right]\right)\prod_{u\in\Z_{n}\backslash\left\{ i_{j}:j\in\Z_{k}\right\} }\mathbf{t}\left[u\right]
\]
\[
=\mathbf{A}\left[n,n\right]\sum_{0\le k\le n}\sum_{0\le i_{0}<\cdots<i_{k-1}<n}\left(-1\right)^{k}\det\left(\mathbf{A}\left[\begin{array}{ccccc}
i_{0} & \cdots & i_{j} & \cdots & i_{k-1}\\
i_{0} & \cdots & i_{j} & \cdots & i_{k-1}
\end{array}\right]\right)\prod_{u\notin\Z_{n}\backslash\left\{ i_{j}:j\in\Z_{k}\right\} }\left(\sum_{v\in\Z_{n+1}}\mathbf{A}\left[u,v\right]\right)
\]
\[
\implies\mathbf{A}\left[n,n\right]\,\det\left(\text{diag}\left(\mathbf{t}\right)-\mathbf{A}\left[:n,:n\right]\right)=
\]
\[
\mathbf{A}\left[n,n\right]{\color{red}\left(\sum_{0\le k\le n}\sum_{0\le i_{0}<\cdots<i_{k-1}<n}\left(-1\right)^{k}\det\left(\mathbf{A}\left[\begin{array}{ccccc}
i_{0} & \cdots & i_{j} & \cdots & i_{k-1}\\
i_{0} & \cdots & i_{j} & \cdots & i_{k-1}
\end{array}\right]\right)\right)}{\color{blue}\left(\prod_{u\in\Z_{n}\backslash\left\{ i_{j}:j\in\Z_{k}\right\} }\sum_{v\in\Z_{n+1}}\mathbf{A}\left[u,v\right]\right)},
\]
\[
\implies\mathbf{A}\left[n,n\right]\det\left(\text{diag}\left(\mathbf{t}\right)-\mathbf{A}\left[:n,:n\right]\right)=
\]
\[
\mathbf{A}\left[n,n\right]{\color{red}\left(\sum_{0\le k\le n}\sum_{0\le i_{0}<\cdots<i_{k-1}<n}\sum_{\sigma\in\text{S}_{k}}\left(-1\right)^{\text{\#\ensuremath{\text{cycles in }G_{\sigma}}}}\prod_{j\in\Z_{k}}\mathbf{A}\left[i_{j},i_{\sigma\left(j\right)}\right]\right)}{\color{blue}\left(\prod_{u\in\Z_{n}\backslash\left\{ i_{j}:j\in\Z_{k}\right\} }\sum_{v\in\Z_{n+1}}\mathbf{A}\left[u,v\right]\right)}.
\]
Every term in the fully expanded expression corresponds to a functional
directed graph prescribed by the product of edge variables in the
term, some of which are colored ${\color{red}\text{Red}}$ in order
to record the fact that they are edges of a spanning unions of cycles
arising from the determinant factor and some others edge variables
are colored ${\color{blue}\text{Blue}}$ in order to record the fact
that they are edges of a subgraph of a functional directed graph arising
from product of some entries of the vector $\mathbf{t}$. The Zeilberger
pairing argument amounts to consider each graph described by each
term summand in the expanded form of the fully expanded polynomial
$\mathbf{A}\left[n,n\right]\,\det\left(\text{diag}\left(\mathbf{t}\right)-\mathbf{A}\left[:n,:n\right]\right)$
of the form 
\[
\mathbf{A}\left[n,n\right]{\color{red}\left(\left(-1\right)^{\text{\#\ensuremath{\text{cycles in }G_{\sigma}}}}\prod_{j\in\Z_{k}}\mathbf{A}\left[i_{j},i_{\sigma\left(j\right)}\right]\right)}{\color{blue}\left(\prod_{u\in\Z_{n}\backslash\left\{ i_{j}:j\in\Z_{k}\right\} }\mathbf{A}\left[u,f\left(u\right)\right]\right)},
\]
for some choice of 
\[
\left\{ i_{j}:j\in\Z_{k}\right\} \subset\Z_{n},\:\sigma\in\text{S}_{k},\;\text{ and }\;f\in\Z_{n+1}^{\left(\Z_{n}\backslash\left\{ i_{j}:j\in\Z_{k}\right\} \right)}
\]
to be paired up with a different functional directed graph on the
same edge set such that every edge variable in the cycle which contains
the smallest vertex label $<n$ ( among all cycles in the selected
graph to be paired up ) switch colors from the red color to the blue
color if the edge variables in the cycle ( containing the smallest
vertex label $<n$ ) was originally colored red or vice-versa if the
cycle which contains the smallest vertex label $<n$ ( among all cycles
in the selected graph to be paired up ) is originally blue. As a result
every graph having at least one blue or red cycle is paired up. The
only graphs that are not paired up are functional trees rooted at
$n$. Consequently coefficients of paired up terms must vanish for
they differ in sign.\\
Note that the graph corresponding to the blue edges is not necessarily
a functional graph in that there may be vertices having out-degree
zero (since the outgoing edges for these vertices may lie in the red
set). The desired claim therefore follows by setting $\mathbf{A}=\mathbf{X}$
and completes the proof.\end{proof}

\begin{prop} The polynomial $P_{\mathbf{X}}\left(x\right)$ has the
following properties:
\begin{itemize}
\item[i.] The lowest-degree term in $P_{\mathbf{X}}\left(x\right)$ is $x^{n(n-1)+1}$. 
\item[ii.] If $n$ is even then $P_{\mathbf{X}}\left(x\right)$ has degree $n^{n-1}+2\sum\limits _{0<i<\frac{n}{2}-1}n^{\frac{n}{2}+i}$.\\
 Otherwise, $P_{\mathbf{X}}\left(x\right)$ has degree $n^{(n-1)/2}+n^{n-1}+2\sum\limits _{0<i<\frac{n-1}{2}-1}n^{\frac{n-1}{2}+i}$. 
\end{itemize}
\end{prop}

\begin{proof} (i) Since functional trees have exactly one root, then
$f(i)\neq i$ for all but one value of $i$.

As such, each term in $P_{\mathbf{X}}\left(x\right)$ has an exponent
that is equal to $n^{0}+(n-1)n^{1}=n(n-1)+1$ or larger. The minimum
is attained with the functional path associated with $f(i)=\max(0,i-1)$.
$\square$\\

(ii) The proof is similar to the corresponding result in Proposition~\ref{prop-f_X-properties},
with the slight modification that one of the edges in the two-cycle
between the central vertices is replaced with a loop.

If $n$ is odd, this means the new induced subtractive edge label
sequence is 
\[
\left\{ \frac{n-1}{2},\frac{n-1}{2}+1,\frac{n-1}{2}+1,\frac{n-1}{2}+2,\frac{n-1}{2}+2,\dots,n-3,n-3,n-2,n-2,n-1,0\right\} .
\]

This yields an exponent of $n^{(n-1)/2}+n-1+\sum\limits _{0<i<\frac{n-1}{2}-1}2n^{\frac{n-1}{2}+i}$.

If $n$ is even the new sequence is 
\[
\left\{ \frac{n}{2},\frac{n}{2},\frac{n}{2}+1,\frac{n}{2}+1,\dots,n-3,n-3,n-2,n-2,n-1,0\right\} .
\]
This yields an exponent of $n-1+\sum\limits _{0<i<\frac{n}{2}-1}2n^{\frac{n}{2}+i}$.
\end{proof} 

\section{Graceful expansion labeling algorithm}

Given the graceful expansion parametrization of a functional directed
graph we describe a procedure for determining the graceful expansion
of every graceful functional directed graphs at edge distance one
from the the input functional graph. Functional directed graphs $G_{f}$,
$G_{g}$ associated with $f,g\in\Z_{n}^{\Z_{n}}$ are at edge edit
distance at most $k$ from one another if 
\begin{equation}
\exists\,\sigma\in\nicefrac{\text{S}_{n}}{\text{Aut}G_{f}}\:\text{and}\,\begin{array}{c}
T\subset\Z_{n}\\
\left|T\right|=k
\end{array}\text{ s.t. }\,g\left(i\right)=\sigma f\sigma^{\left(-1\right)}\left(i\right),\quad\forall\,i\in\Z_{n}\backslash T\label{Edge_Distance_at_most_one}
\end{equation}
this relation is symmetric reflexive but not transitive.\\
 \\
 \textbf{Algorithm 4.0.} : Let $G_{f}$ denote the input functional
directed graph of $f\in\Z_{n}^{\Z_{n}},$ whose known parametric graceful
expansion is given by 
\[
f(i)=\s_{\gamma}\left(\s_{\gamma}^{\left(-1\right)}(i)+(-1)^{p_{\gamma}\left(\s_{\gamma}^{\left(-1\right)}(i)\right)}\gamma\left(\s_{\gamma}^{\left(-1\right)}(i)\right)\right),\quad\forall\,i\in\Z_{n}.
\]
For every $\gamma\in\mathcal{S}\subset$ S$_{n}$ perform all possible
valid single sign changes defined by selecting some integer $j\in\Z_{n}$
\[
p_{j,\gamma}^{\prime}\left(i\right)=\begin{cases}
\begin{array}{cc}
\left(1+p_{\gamma}\left(i\right)\right)\text{mod }2 & \text{ if }i=j\\
\\
p_{\gamma}\left(i\right) & \text{otherwise}
\end{array},\end{cases}
\]
to obtain the graceful expansion of gracefully labeled functional
directed graphs expressed by 
\[
\left(\text{id}+\left(-1\right)^{p_{\gamma}^{\prime}\left(\text{id}\right)}\gamma\left(\text{id}\right)\right)\in\Z_{n}^{\Z_{n}}.
\]
which is associated with a functional directed graphs at edge edit
distance at most one from $G_{f}$. By the proof argument of Prop.
2.1 there are at most $\left\lceil \frac{n-1}{2}\right\rceil $ possible
sign changes for each valid $\gamma\in\mathcal{S}$. \\
\\
As illustration consider $f\in\Z_{n}^{\Z_{n}}$ given by 
\[
f\left(i\right)=0\quad\forall\:i\in\Z_{n},
\]
The parametric graceful expansion of $f$ expressed in terms of the
parameter $\gamma\in\mathcal{S}\subset$ S$_{n}$ such that 
\[
\mathcal{S}=\left\{ \text{id},n-1-\text{id}\right\} .
\]
\[
p_{\text{id}}\left(i\right)=1\:\text{ and }\:p_{n-1-\text{id}}\left(i\right)=0\quad\forall\:i\in\Z_{5}.
\]
\[
\sigma_{\text{id}}=\text{id}\:\text{ and }\:\sigma_{n-1-\text{id}}=n-1-\text{id}.
\]
Functional graphs at edge distance $0$ from $G_{f}$ are elements
of GrL$\left(G_{f}\right)$ given by 
\[
\text{GrL}\left(G_{f}\right)=\left\{ G_{f},G_{n-1-f}\right\} .
\]
Furthermore each integer $j\in\left\{ 1,\cdots,\left\lfloor \frac{n-1}{2}\right\rfloor \right\} $
yields a new gracefully labeled functional directed graphs at edge
distance at most one from $G_{f}$ associated with functions 
\[
g_{\text{id},j}\left(i\right)=i+(-1)^{p_{j,\text{id}}^{\prime}\left(i\right)}\,i\quad\forall\:i\in\Z_{n},
\]
and each integer $j\in\left\{ \left\lceil \frac{n-1}{2}\right\rceil ,\cdots,n-1\right\} $
yields a new gracefully labeled functional directed graphs at edge
distance at most one from $G_{f}$ associated with functions 
\[
g_{n-1-\text{id},j}\left(i\right)=i+(-1)^{p_{j,n-1-\text{id}}^{\prime}\left(i\right)}\,\left(n-1-i\right)\quad\forall\:i\in\Z_{n},
\]
 \begin{prop}Algorithm 4.0 identifies all gracefully labeled functional
directed graphs at edge edit distance at most one from $G_{f}$.

\end{prop}

\begin{proof} We prove the claim by contradiction. Assume for the
sake of establishing a contradiction that the desired claim is false.
It would follow that there is some gracefully labeled functional directed
graph $G_{g}$ at edge distance at most one from some input functional
directed graph $G_{f}$ which is not identified by Algorithm 4.0.
\[
G_{g}\notin\text{GrL}\left(G_{f}\right),
\]
for otherwise $G_{g}$ would be deduced from the known parametric
graceful expansion of $f$. Consequently $G_{g}$ is not isomorphic
to $G_{f}$. From the fact that $G_{g}$ is at edge edit distance
at most one from $G_{f}$ it follows that 
\[
\exists\,\gamma\in\mathcal{S},\,\sigma\in\nicefrac{\text{S}_{n}}{\text{Aut}G_{f}}\:\text{ and }\,j\in\Z_{n}\text{ such that }\,g\left(i\right)=\sigma f\sigma^{\left(-1\right)}\left(i\right)\:\forall\,i\in\Z_{n}\backslash\left\{ j\right\} ,
\]
where $G_{\sigma f\sigma^{\left(-1\right)}}\in$ GrL$\left(G_{f}\right)$
\[
\implies\left|g\left(j\right)-j\right|=\gamma\left(j\right)=\left|\sigma f\sigma^{\left(-1\right)}\left(j\right)-j\right|,
\]
which contradicts the premise that $g$ is not obtained by a single
sign change in the graceful expansion of $f$ therefore concludes
the proof. \end{proof} \textbf{}\\
\textbf{Conjecture 4.2.} : Any induced subtractive edge label sequence
of an identically constant functions in $\Z_{n}^{\Z_{n}}$ appears
among induced subtractive edge label sequences of graphs in the isomorphism
class of any functional tree in $\Z_{n}^{\Z_{n}}$.\\
 \\
 \emph{Acknowledgement}: The first author would like to thank Noga
Alon for introducing him to the subject. We are grateful to Harry
Crane, Mark Daniel Ward, Yuval Filmus, Andrei Gabrielov, Edward R.
Scheinerman and Jeanine Gnang for insightful discussions and suggestions.\\
 \\
 This research was partially supported by NSF-DMS grants 1603823,
1604458, and 1604773, ``Collaborative Research: Rocky Mountain -
Great Plains Graduate Research Workshops in Combinatorics'' and the
NSA grant H98230-18-1-0017.

\end{document}